\documentclass[12pt,a4paper,english]{article}
\usepackage[T1]{fontenc}
\usepackage[latin9]{inputenc}
\usepackage{amsmath}
\usepackage{amssymb}

\makeatletter

\usepackage{amsthm}\usepackage{fullpage}\usepackage{makeidx}\usepackage{amsfonts}\usepackage{latexsym}\makeindex

\def\C{{\mathbb{C}}}

\theoremstyle{definition}
\newtheorem{lemma}{Lemma}[section]
\newtheorem{theorem}[lemma]{Theorem}

\newtheorem{proposition}[lemma]{Proposition}
\newtheorem{definition}[lemma]{Definition}

\newtheorem{remark}[lemma]{Remark}
\usepackage{times}

\usepackage{babel}

\usepackage{babel}

\usepackage[blocks]{authblk}% The option is for block layout
\title{Cyclic permutations of lattice vertex operator algebras
}
\author{Chongying Dong\footnote{Supported by NSF grant DMS-1404741, NSA grant H98230-14-1-0118 and China NSF grant 11371261}}
\affil{Department of Mathematics, University of
California, Santa Cruz, CA 95064 USA}
\author{Feng Xu\footnote{Partially supported by China NSF 11471064}}
\author{Nina Yu}
\affil{Department of Mathematics, University of California, Riverside, CA 92521 USA}

\usepackage{babel}

\usepackage{babel}

\makeatother

\usepackage{babel}
\begin{document}
\maketitle
\begin{abstract}
The irreducible modules of the 2-cycle permutation orbifold models
of lattice vertex operator algebras of rank 1 are classified, the
quantum dimensions of irreducible modules and the fusion rules are
determined.
\end{abstract}

\section{Introduction}

This paper is about the permutation orbifolds of the rank one lattice
vertex operator algebras under the action of the 2-cycle. The permutation
orbifolds study the permutations actions on the tensor products of
vertex operator algebras. Namely, given a vertex operator algebra
$V$, then tensor product $V^{\otimes n}$ of $n$-copies of $V$
as a $\mathbb{C}$-vector space naturally has a vertex operator algebra
structure \cite{FHL}. Any element $\sigma$ of the symmetric group
$S_{n}$ acts on $V^{\otimes n}$ in the obvious way, and gives an
automorphism of $V^{\otimes n}$ of finite order. The fixed points
set $\left(V^{\otimes n}\right)^{\sigma}=\left\{ v\in V^{\otimes n}|\sigma\left(v\right)=v\right\} $
is a vertex operator subalgebra which is called a $\sigma$-permutation
orbifold model.

A systematic study was started in \cite{BHS} for cyclic permutation
orbifolds for affine algebras and Virasoro algebras, where the twisted
sectors, genus one characters and their modular transformations, as
well as the fusion rules were treated. The genus one characters and
modular transformation properties of permutation orbifolds of arbitrary
rational conformal field theories was presented in \cite{Ba}. The
twisted sectors or twisted modules for the permutation orbifolds of
an arbitrary vertex operator algebra orbifold models were constructed
in \cite{BDM}. In the case that $V$ is a lattice vertex operator
algebra, tensor products of $V$ is another lattice vertex operator
algebra again. So the twisted sectors for permutation orbifolds in
this case were given in \cite{Le}, \cite{DL2} already. See \cite{BHL}
for the details on this. The $C_{2}$-cofiniteness of permutation
orbifolds and general cyclic orbifolds have been obtained recently
in \cite{A3}, \cite{A4} and \cite{M}. Conformal nets approach to
permutation orbifolds have been studied in \cite{KLX} where irreducible
representations of the cyclic orbifold are determined and fusion rules
were given for $n=2.$

However, representation theory of permutation orbifolds from the point
of view of vertex operator algebras has not been well understood,
such as rationality, classification of irreducible modules, and fusion
rules. We study the representations of order 2 cyclic permutation
orbifold models for rank one lattice vertex operator algebras in this
paper. In particular we compute the quantum dimensions of the irreducible
modules and determine the fusions. It turns out that such permutation
orbifold can be realized as a simple current extension of the tensor
product of a lattice vertex operator algebra $V_{L}$ with the fixed
point vertex operator algebra $V_{L}^{+}$ under the $-1$ automorphism.
So the main idea is to use the representations of $V_{L}$ and $V_{L}^{+}$
\cite{D1}, \cite{D2}, \cite{DN2}, \cite{AD}, \cite{ADL}, \cite{A2},
\cite{DJL} to understand the representation theory of the permutation
orbifolds such as classification of irreducible modules and the fusion
rules. The results of this paper agree with that of \cite{KLX} in the setting of conformal nets.

The paper is organized as follows: $ $\S2 and \S3 are preliminaries
on the vertex operator algebras theory framework. In these sections
we give some basic notions that appear in this paper and recall construction
of the lattice vertex operator algebras $V_{\mathbb{Z}\alpha}$ and
$V_{\mathbb{Z}\alpha}^{+}$. In \S4 we study $\left(V_{\mathbb{Z}\alpha}\otimes V_{\mathbb{Z}\alpha}\right)^{\mathbb{Z}_{2}}$,
the 2 cyclic permutation orbifold models for rank one lattice vertex
operator algebras. We prove the rationality of $\left(V_{\mathbb{Z}\alpha}\otimes V_{\mathbb{Z}\alpha}\right)^{\mathbb{Z}_{2}}.$
The classification of the irreducible modules of $\left(V_{\mathbb{Z}\alpha}\otimes V_{\mathbb{Z}\alpha}\right)^{\mathbb{Z}_{2}}$
is given in \S5. The quantum dimensions of all irreducible modules
of $\left(V_{\mathbb{Z}\alpha}\otimes V_{\mathbb{Z}\alpha}\right)^{\mathbb{Z}_{2}}$
are computed in \S6. We apply results from the previous sections
to determine all fusion products in \S7.

\section{Preliminaries}

Let $\left(V,Y,\mathbf{1},\omega\right)$ be a vertex operator algebra
\cite{Bo}, \cite{FLM}. First we recall notions of an automorphism
$g$ of finite order of a vertex operator algebra and its $g$-twisted
modules \cite{FLM}, \cite{DLM2}.

\begin{definition} An \emph{automorphism} $g$ of a vertex operator
algebra $V$ is a linear isomorphism of $V$ satisfying $g\left(\omega\right)=\omega$
and $gY\left(v,z\right)g^{-1}=Y\left(gv,z\right)$ for any $v\in V$.
We denote by $\mbox{Aut}\left(V\right)$ the group of all automorphism
of $V$.

For a subgroup $G\le\mbox{Aut}\left(V\right)$ the fixed point set
$V^{G}=\left\{ v\in V|g\left(v\right)=v,\forall g\in G\right\} $
has a vertex operator algebra structure. \end{definition} Let $g$
be an automorphism of a vertex operator algebra $V$ of order $T$.
Denote the decomposition of $V$ into eigenspaces of $g$ as:

\[
V=\oplus_{r\in\mathbb{Z}/T\text{\ensuremath{\mathbb{Z}}}}V^{r}
\]
where $V^{r}=\left\{ v\in V|gv=e^{2\pi ir/T}v\right\} $.

\begin{definition}A \emph{weak $g$-twisted $V$-module} $M$ is
a vector space with a linear map
\[
Y_{M}:V\to\left(\text{End}M\right)\{z\}
\]

\[
v\mapsto Y_{M}\left(v,z\right)=\sum_{n\in\mathbb{Q}}v_{n}z^{-n-1}\ \left(v_{n}\in\mbox{End}M\right)
\]

which satisfies the following: for all $0\le r\le T-1$, $u\in V^{r}$,
$v\in V$, $w\in M$,

\[
Y_{M}\left(u,z\right)=\sum_{n\in-\frac{r}{T}+\mathbb{Z}}u_{n}z^{-n-1},
\]

\[
u_{l}w=0\ for\ l\gg0,
\]

\[
Y_{M}\left(\mathbf{1},z\right)=Id_{M},
\]

\[
z_{0}^{-1}\text{\ensuremath{\delta}}\left(\frac{z_{1}-z_{2}}{z_{0}}\right)Y_{M}\left(u,z_{1}\right)Y_{M}\left(v,z_{2}\right)-z_{0}^{-1}\delta\left(\frac{z_{2}-z_{1}}{-z_{0}}\right)Y_{M}\left(v,z_{2}\right)Y_{M}\left(u,z_{1}\right)
\]

\[
z_{2}^{-1}\left(\frac{z_{1}-z_{0}}{z_{2}}\right)^{-r/T}\delta\left(\frac{z_{1}-z_{0}}{z_{2}}\right)Y_{M}\left(Y\left(u,z_{0}\right)v,z_{2}\right),
\]
where $\delta\left(z\right)=\sum_{n\in\mathbb{Z}}z^{n}$. \end{definition}

\begin{definition}

A $g$-\emph{twisted $V$-module} is a weak $g$-twisted $V$-module
$M$ which carries a $\mathbb{C}$-grading induced by the spectrum
of $L(0)$ where $L(0)$ is the component operator of $Y(\omega,z)=\sum_{n\in\mathbb{Z}}L(n)z^{-n-2}.$
That is, we have $M=\bigoplus_{\lambda\in\mathbb{C}}M_{\lambda},$
where $M_{\lambda}=\{w\in M|L(0)w=\lambda w\}$. Moreover, $\dim M_{\lambda}$
is finite and for fixed $\lambda,$ $M_{\frac{n}{T}+\lambda}=0$ for
all small enough integers $n.$

\end{definition}

\begin{definition}An \emph{admissible $g$-twisted $V$-module} $M=\oplus_{n\in\frac{1}{T}\mathbb{Z}_{+}}M\left(n\right)$
is a $\frac{1}{T}\mathbb{Z}_{+}$-graded weak $g$-twisted module
such that $u_{m}M\left(n\right)\subset M\left(\mbox{wt}u-m-1+n\right)$
for homogeneous $u\in V$ and $m,n\in\frac{1}{T}\mathbb{Z}.$ $ $

\end{definition}

If $g=Id_{V}$ we have the notions of weak, ordinary and admissible
$V$-modules \cite{DLM2}.

\begin{definition}A vertex operator algebra $V$ is called \emph{$g$-rational}
if the admissible $g$-twisted module category is semisimple. $V$
is called \emph{rational} if $V$ is $1$-rational. \end{definition}

It is proved in \cite{DLM2} that if $V$ is a $g$-rational vertex
operator algebra, then there are only finitely many irreducible admissible
$g$-twisted $V$-modules up to isomorphism and and any irreducible
admissible $g$-twisted $V$-module is ordinary.

\begin{definition} A vertex operator algebra $V$ is said to be $C_{2}$-cofinite
if $V/C_{2}(V)$ is finite dimensional, where $C_{2}(V)=\langle v_{-2}u|v,u\in V\rangle.$
\end{definition}

\begin{remark} If $V$ is a $C_{2}$-cofinite vertex operator algebra,
then $V$ has only finitely many irreducible admissible modules up
to isomorphism \cite{DLM2,Li}. \end{remark}

We need the contragredient module \cite{FHL} in this paper. \begin{definition}
Let $M=\bigoplus_{n\in\frac{1}{T}\mathbb{Z}_{+}}M(n)$ be an admissible
$g$-twisted $V$-module, the contragredient module $M'$ is defined
as follows:
\[
M'=\bigoplus_{n\in\frac{1}{T}\mathbb{Z}_{+}}M(n)^{*},
\]
where $M(n)^{*}=\mbox{Hom}_{\mathbb{C}}(M(n),\mathbb{C}).$ The vertex
operator $Y_{M'}(v,z)$ is defined for $v\in V$ via
\begin{eqnarray*}
\langle Y_{M'}(v,z)f,u\rangle= & \langle f,Y_{M}(e^{zL(1)}(-z^{-2})^{L(0)}v,z^{-1})u\rangle
\end{eqnarray*}
where $\langle f,w\rangle=f(w)$ is the natural paring $M'\times M\to\mathbb{C}.$
\end{definition}

\begin{remark} 1. $(M',Y_{M'})$ is an admissible $g^{-1}$-twisted
$V$-module \cite{FHL}.

2. We can also define the contragredient module $M'$ for a $g$-twisted
$V$-module $M.$ In this case, $M'$ is a $g^{-1}$-twisted $V$-module.
Moreover, $M$ is irreducible if and only if $M'$ is irreducible.

\end{remark}

Here are the definition of intertwining operators and fusion rules
\cite{FHL}.

\begin{definition} Let $(V,\ Y)$ be a vertex operator algebra and
let $(W^{1},\ Y^{1}),\ (W^{2},\ Y^{2})$ and $(W^{3},\ Y^{3})$ be
$V$-modules. An intertwining operator of type $\left(\begin{array}{c}
W^{1}\\
W^{2\ }W^{3}
\end{array}\right)$ is a linear map
\[
I(\cdot,\ z):\ W^{2}\to\text{\ensuremath{\mbox{Hom}(W^{3},\ W^{1})\{z\}}}
\]

\[
u\to I(u,\ z)=\sum_{n\in\mathbb{Q}}u_{n}z^{-n-1}
\]
satisfying:

(1) for any $u\in W^{2}$ and $v\in W^{3}$, $u_{n}v=0$ for $n$
sufficiently large;

(2) $I(L_{-1}v,\ z)=(\frac{d}{dz})I(v,\ z)$;

(3) (Jacobi Identity) for any $u\in V,\ v\in W^{2}$

\[
z_{0}^{-1}\delta\left(\frac{z_{1}-z_{2}}{z_{0}}\right)Y^{1}(u,\ z_{1})I(v,\ z_{2})-z_{0}^{-1}\delta\left(\frac{-z_{2}+z_{1}}{z_{0}}\right)I(v,\ z_{2})Y^{3}(u,\ z_{1})
\]
\[
=z_{2}^{-1}\left(\frac{z_{1}-z_{0}}{z_{2}}\right)I(Y^{2}(u,\ z_{0})v,\ z_{2}).
\]

We denote the space of all intertwining operators of type $\left(\begin{array}{c}
W^{1}\\
W^{2}\ W^{3}
\end{array}\right)$  by $I_{V}\left(\begin{array}{c}
W^{1}\\
W^{2}\ W^{3}
\end{array}\right).$ Let $N_{W^{2},\ W^{3}}^{W^{1}}=\dim I_{V}\left(\begin{array}{c}
W^{1}\\
W^{2}\ W^{3}
\end{array}\right)$. These integers $N_{W^{2},\ W^{3}}^{W^{1}}$ are usually called the
\textbf{fusion rules}. \end{definition}

\begin{definition} Let $V$ be a vertex operator algebra, and $W^{1},$
$W^{2}$ be two $V$-modules. A module $(W,I)$, where $I\in I_{V}\left(\begin{array}{c}
\ \ W\ \\
W^{1}\ \ W^{2}
\end{array}\right),$ is called a fusion product of $W^{1}$ and $W^{2}$ if for any $V$-module
$M$ and $\mathcal{Y}\in I_{V}\left(\begin{array}{c}
\ \ M\ \\
W^{1}\ \ W^{2}
\end{array}\right),$ there is a unique $V$-module homomorphism $f:W\rightarrow M,$ such
that $\mathcal{Y}=f\circ I.$ As usual, we denote $(W,I)$ by $W^{1}\boxtimes_{V}W^{2}.$
\end{definition}

It is well known that if $V$ is rational, then the fusion product
exists. We shall often consider the fusion product
\[
W^{1}\boxtimes_{V}W^{2}=\sum_{W}N_{W^{1},\ W^{2}}^{W}W
\]
where $W$ runs over the set of equivalence classes of irreducible
$V$-modules.

The following symmetry property of fusion rules is well known \cite{FHL}.

\begin{proposition}\label{fusion rule symmmetry property} Let $W^{i}$
$\left(i=1,2,3\right)$ be $V$-modules. Then

\[
\dim I_{V}\left(_{W^{1}W^{2}}^{\ \ W^{3}}\right)=\dim I_{V}\left(_{W^{2}W^{1}}^{\ \ W^{3}}\right),\dim I_{V}\left(_{W^{1}W^{2}}^{\ \ W^{3}}\right)=\dim I_{V}\left(_{W^{1}\left(\ W^{3}\right)'}^{\ \ \left(W^{2}\right)'}\right).
\]

\end{proposition}

\begin{definition} Let $V$ be a simple vertex operator algebra.
A simple $V$-module $M$ is called \emph{a simple current} if for
any irreducible $V$-module $W$, $W\boxtimes M$ exists and is also
a simple $V$-module. \end{definition}

Let $D$ be a finite abelian group and assume that we have a set of
irreducible simple current $V^{0}$-modules $\left\{ V^{\alpha}|\alpha\in D\right\} $
indexed by $D$.

\begin{definition} An extension $V_{D}=\oplus_{\alpha\in D}V^{\alpha}$
of $V^{0}$ is called a $D$-\emph{graded simple current extension
}if $V_{D}$ carries a structure of a simple vertex operator algebra
such that $Y\left(u^{\alpha},z\right)u^{\beta}\in V^{\alpha+\beta}\left(\left(z\right)\right)$
for any $u^{\alpha}\in V^{\alpha}$ and $xu^{\beta}\in V^{\beta}$.
\end{definition}

\section{The vertex operator algebra $V_{L}$ and $V_{L}^{+}$}

Let $L=\mathbb{Z}\alpha$ be a positive definite even lattice of rank
one with $\left\langle \alpha,\alpha\right\rangle =2k$ for some positive
integer $k$. In this section, we review the construction of rank
one lattice vertex operator algebra from \cite{FLM}. Then we give
some related results about $V_{L}$ and $V_{L}^{+}$ \cite{A1,A2,AD,ADL,DN1,DN2,DN3,DJL,FLM}.

Set $\mathfrak{h}=L\otimes_{\mathbb{Z}}\mathbb{C}$ and extend $\left\langle \cdot,\cdot\right\rangle $
to a $\mathbb{C}$-bilinear form on $\mathfrak{h}$. Let $\mathfrak{h}=\mathbb{C}\otimes_{\mathbb{Z}}L$
and extend $\left\langle \cdot,\cdot\right\rangle $ to a $\mathbb{C}$-bilinear
form on $\mathfrak{h}$. Let $\hat{\mathfrak{h}}=\mathbb{C}\otimes_{\mathbb{Z}}L\oplus\mathbb{C}K$
be the affine Lie algebra associated to the abelian Lie algebra $\mathfrak{h}$
so that
\[
\left[\alpha\left(m\right),\alpha\left(n\right)\right]=2km\delta_{m,-n}K\ \ and\ \ \left[K,\hat{\mathfrak{h}}\right]=0
\]
for any $m,n\in\mathbb{Z}$ where $\alpha\left(m\right)=\alpha\otimes t^{m}$.
Then $\hat{\mathfrak{h}}_{\ge0}=\mathbb{C}\left[t\right]\otimes\mathfrak{h}\oplus\mathbb{C}K$
is a commutative subalgebra. For any $\lambda\in\mathfrak{h},$ we
define a one $\hat{\mathfrak{h}}_{\ge0}$-module $\mathbb{C}e^{\lambda}$
such that $\alpha\left(m\right)e^{\lambda}=\left\langle \lambda,\alpha\right\rangle \delta_{m,0}e^{\lambda}$
and $K\cdot e^{\lambda}=e^{\lambda}$ for $m\ge0$. We denote by
\[
M\left(1,\lambda\right)=U\left(\hat{\mathfrak{h}}\right)\otimes_{U\left(\hat{\mathfrak{h}}_{\ge0}\right)}\otimes\mathbb{C}e^{\lambda}\cong S\left(t^{-1}\mathbb{C}\left[t^{-1}\right]\right)
\]
the $\hat{\mathfrak{h}}$-module induced from $\hat{\mathfrak{h}}_{\ge0}$-module
$\mathbb{C}e^{\lambda}$. Set $M\left(1\right)=M\left(1,0\right)$.
Then there exists a linear map $Y:M\left(1\right)\to\mbox{End}M\left(1\right)\left[\left[z,z^{-1}\right]\right]$
such that $\left(M\left(1\right),Y,\mathbf{1},\omega\right)$ carries
a simple vertex operator algebra structure and $M\left(1,\lambda\right)$
becomes an irreducible $M\left(1\right)$-module for any $\lambda\in\mathfrak{h}$
\cite{FLM}. Let $\mathbb{C}\left[L\right]$ be the group algebra
of $L$ with a basis $e^{\beta}$ for $\beta\in L$. The lattice vertex
operator algebra associated to $L$ is given by
\[
V_{L}=M\left(1\right)\otimes\mathbb{C}\left[L\right].
\]
The dual lattice $L^{\circ}$ of $L$ is
\[
L^{\circ}=\left\{ \lambda\in\mathfrak{h}|\left\langle \alpha,\lambda\right\rangle \in\mathbb{Z}\right\} =\frac{1}{2k}L.
\]

Then $L^{\circ}=\cup_{i=-k+1}^{k}\left(L+\lambda_{i}\right)$ is the
coset decomposition with $\lambda_{i}=\frac{i}{2k}\alpha$. Set $\mathbb{C}\left[L+\lambda_{i}\right]=\oplus_{\beta\in L}\mathbb{C}e^{\beta+\lambda_{i}}$.
Then each $\mathbb{C}\left[L+\lambda_{i}\right]$ is an $L$-submodule.
Set $V_{L+\lambda_{i}}=M\left(1\right)\otimes\mathbb{C}\left[L+\lambda_{i}\right]$.
Then $V_{L}$ is a rational vertex operator algebra and $V_{L+\lambda_{i}}$
for $i=0,1,\cdots,2k-1$ are the irreducible modules for $V_{L}$
\cite{Bo,FLM,D1,DLM1}. Define a linear isomorphism $\theta:V_{L+\lambda_{i}}\to V_{L-\lambda_{i}}$
for $i\in\left\{ 0,1,\cdots,2k-1\right\} $ by
\[
\theta\left(\alpha\left(-n_{1}\right)\alpha\left(-n_{2}\right)\cdots\alpha\left(-n_{k}\right)\otimes e^{\beta+\lambda_{i}}\right)=\left(-1\right)^{k}\alpha\left(-n_{1}\right)\cdots\alpha\left(-n_{k}\right)\otimes e^{-\beta-\lambda_{i}}
\]
where $n_{j}>0$ and $\beta\in L$. Then $\theta$ defines a linear
isomorphism from $V_{L^{\circ}}=M\left(1\right)\otimes\mathbb{C}\left[L^{\circ}\right]$
to itself such that
\[
\theta\left(Y\left(u,z\right)v\right)=Y\left(\theta u,z\right)\theta v
\]
for $u\in V_{L}$ and $v\in V_{L^{\circ}}$. In particular, $\theta$
is an automorphism of $V_{L}$ which induces an automorphism of $M\left(1\right)$.

For any $\theta$-stable subspace $U$ of $V_{L^{\circ}}$, let $U^{\pm}$
be the $\pm1$-eigenspace of $U$ for $\theta$. Then $V_{L}^{+}$
is a simple vertex operator algebra.

Recall the $\theta$-twisted Heisenberg algebra $\mathfrak{h}\left[-1\right]$
and its irreducible module $M\left(1\right)\left(\theta\right)$ from
\cite{FLM}. Note that $L/2L$ has two irreducible modules $T_{s}=\C$
such that $\alpha$ acts as $(-1)^{s}$ for $s=1,2.$ Then $V_{L}^{T_{s}}=M\left(1\right)\left(\theta\right)\otimes T_{{s}}$
is an irreducible $\theta$-twisted module of $V_{L}$ \cite{FLM,D2}.
Define actions of $\theta$ on $M\left(1\right)\left(\theta\right)$
and $V_{L}^{T_{{s}}}$ by
\[
\theta\left(\alpha\left(-n_{1}\right)\cdots\alpha\left(-n_{k}\right)\right)=\left(-1\right)^{k}\alpha\left(-n_{1}\right)\cdots\alpha\left(-n_{k}\right)
\]

\[
\theta\left(\alpha\left(-n_{1}\right)\cdots\alpha\left(-n_{k}\right)\otimes t\right)=\left(-1\right)^{k}\alpha\left(-n_{1}\right)\cdots\alpha\left(-n_{k}\right)\otimes t
\]
for $n_{j}\in\frac{1}{2}+\mathbb{Z}_{+}$and $t\in T_{{s}}$. We denote
the $\pm1$-eigenspaces of $V_{L}^{T_{{s}}}$ under $\theta$ by $\left(V_{L}^{T_{{s}}}\right)^{\pm}$.
Then we have have the following results:

\begin{theorem} Any irreducible $V_{L}^{+}$-module is isomorphic
to one of the following modules:
\[
V_{L}^{\pm},V_{\lambda_{i}+L},V_{\lambda_{k}+L}^{\pm},\left(V_{L}^{T_{{s}}}\right)^{\pm}
\]
for $1\le i\le k-1$ and $s=1,2.$ \end{theorem}

\begin{theorem} $V_{L}^{+}$ is rational. \end{theorem}

\begin{remark} The classification of irreducible modules for arbitrary
$V_{L}^{+}$ are obtained in \cite{DN2,AD}. The rationality of $V_{L}^{+}$
is established in \cite{A2} for rank one lattice and in \cite{DJL}
for general case. Fusion rules of all irreducible $V_{L}^{+}$-modules
are given in \cite{A1} for rank one lattice $L$ and in \cite{ADL}
for general case. \end{remark}

\section{Rationality of vertex operator algebra $\left(V_{L}\otimes V_{L}\right)^{\mathbb{Z}_{2}}$}

From now on, we set $V_{L}=V_{\mathbb{Z}\alpha}$ where $\left\langle \alpha,\alpha\right\rangle =2k$
with $k$ a positive integer. The vertex operator algebra $V_{L}\otimes V_{L}$
is isomorphic to the lattice vertex operator algebra $V_{L\oplus L}$,
where $L\oplus L$ is the orthogonal direct sum of two copies of the
lattice $L$. Denote by $\alpha^{1}$, $\alpha^{2}$ the copy of $\alpha$
in the first and second summand of $L\oplus L$ respectively. Let
$\beta_{1}=\alpha^{1}+\alpha^{2}\in L\oplus L$, $\beta_{2}=\alpha^{1}-\alpha^{2}\in L\oplus L$,
then $\left\langle \beta_{i},\beta_{i}\right\rangle =4k$, $i=1,2$
and $\left\langle \beta_{1},\beta_{2}\right\rangle =0$. Thus $\beta_{1}$
and $\beta_{2}$ are orthogonal and $\mathbb{Z}\beta_{1}\cong\mathbb{Z}\beta_{2}\cong\sqrt{2}\mathbb{Z}\alpha$.

Let $L_{0}=\mathbb{Z}\beta_{1}\oplus\mathbb{Z}\beta_{2}$, then
\[
L\oplus L=L_{0}\cup\left(\frac{1}{2}\left(\beta_{1}+\beta_{2}\right)+L_{0}\right).
\]
Since $V_{L_{0}}\cong V_{\mathbb{Z}\beta_{1}}\otimes V_{\mathbb{Z}\beta_{2}}$
as vertex operator algebras and $V_{\frac{1}{2}\left(\beta_{1}+\beta_{2}\right)+L_{0}}\cong V_{\frac{1}{2}\beta_{1}+\mathbb{Z}\beta_{1}}\otimes V_{\frac{1}{2}\beta_{2}+\mathbb{Z}\beta_{2}}$
as $V_{L_{0}}$-modules, we obtain

\[
V_{L\oplus L}\cong V_{L}\otimes V_{L}\cong V_{\mathbb{Z}\beta_{1}}\otimes V_{\mathbb{Z}\beta_{2}}\oplus V_{\frac{1}{2}\beta_{1}+\mathbb{Z}\beta_{1}}\otimes V_{\frac{1}{2}\beta_{2}+\mathbb{Z}\beta_{2}}.
\]

Notice that the 2-cyclic permutation $\sigma\in\mathbb{Z}_{2}$ corresponds
to an automorphism of $V_{L\oplus L}$ lifted from the transpostion
$\left(\delta,\gamma\right)\mapsto\left(\gamma,\delta\right)$ of
$L\oplus L$. In particular, $\sigma$ fixes all vectors in the vertex
operator subalgebra $V_{\mathbb{Z}\beta_{1}}$ of $V_{L\oplus L}$
and acts on $V_{\mathbb{Z}\beta_{2}}$ as the involution $\theta$.
Thus we have the following isomorphism
\begin{equation}
\left(V_{L}\otimes V_{L}\right)^{\mathbb{Z}_{2}}\cong V_{\mathbb{Z}\beta}\otimes V_{\mathbb{Z}\beta}^{+}\oplus V_{\frac{1}{2}\beta+\mathbb{Z}\beta}\otimes V_{\frac{1}{2}\beta+\mathbb{Z}\beta}^{+}\label{Realizaiton of the orbifold}
\end{equation}
where $\left\langle \beta,\beta\right\rangle =4k$.

We need the following result about rationality of a simple current
extension \cite{Y}.

\begin{lemma}\label{simple-current-extension} The simple current
extensions of rational $C_{2}$-cofinite vertex operator algebras
of CFT-type are rational. \end{lemma}

From short, we set $\mathcal{U}=V_{\mathbb{Z}\beta}\otimes V_{\mathbb{Z}\beta}^{+}\oplus V_{\frac{1}{2}\beta+\mathbb{Z}\beta}\otimes V_{\frac{1}{2}\beta+\mathbb{Z}\beta}^{+}$.
\begin{proposition} \label{Rationality} The vertex operator algebra
$\mathcal{U}$ is rational.

\end{proposition}

\begin{proof} Since both the vertex operator algebras $V_{\mathbb{Z}\beta}$
and $V_{\mathbb{Z}\beta}^{+}$ are rational and $C_{2}$-cofinite,
$V_{\mathbb{Z}\beta}\otimes V_{\mathbb{Z}\beta}^{+}$ is a rational
$C_{2}$-cofinite vertex operator algebra \cite{DLM1}. By fusion
rules of irreducible modules of $V_{\mathbb{Z}\beta}$ and $V_{\mathbb{Z}\beta}^{+}$,
we see that $V_{\frac{1}{2}\beta+\mathbb{Z}\beta}\otimes V_{\frac{1}{2}\beta+\mathbb{Z}\beta}^{+}$
is a simple current for the vertex operator algebra $V_{\mathbb{Z}\beta}\otimes V_{\mathbb{Z}\beta}^{+}$.
By Lemma \ref{simple-current-extension}, $\mathcal{U}$ is rational.
\end{proof}

\section{Classification of irreducible modules}

We first decompose each irreducible and irreducbile twisted $V_{L}\otimes V_{L}$-modules
into direct sum of irreducible $(V_{L}\otimes V_{L})^{\mathbb{Z}_{2}}$-modules.
Then we show that this gives a complete irreducible $(V_{L}\otimes V_{L})^{\mathbb{Z}_{2}}$-modules.

Recall that the irreducible $V_{L}$-modules are $V_{L+\lambda_{i}}$
for $i=0,...,2k-1.$ The the irreducible $V_{L}\otimes V_{L}$-modules
are $V_{L+\lambda_{i}}\otimes V_{L+\lambda_{j}}$ for $i,j\in\left\{ 0,1,\cdots,2k-1\right\} .$
If $i\not=j$, then $V_{L+\lambda_{i}}\otimes V_{L+\lambda_{j}}$
and $V_{L+\lambda_{j}}\otimes V_{L+\lambda_{i}}$ are isomorphic irreducible
$(V_{L}\otimes V_{L})^{\mathbb{Z}_{2}}$ \cite{DM}, \cite{DY}. %carries a representation of the $\mathbb{Z}_{2}$ permutation symmetry,
%which can be decomposed into a direct sum of two irreducible modules
%of the orbifold algebra by \cite{DY}. One of these submodules contains
%the vectors which are symmetric under the permutation and the other
%contains those antisymmetric ones. The symmetric and antisymmetric
%vectors can be written in the form
%\[
%x\otimes y+y\otimes x,\ x\otimes y-y\otimes x
%\]
%respectively, where $x\in V_{L+\lambda_{i}}$ and $y\in V_{L+\lambda_{j}}$.
%Notice that these two modules are isomorphic under the map:
%\[
%x\otimes y+y\otimes x\mapsto x\otimes y-y\otimes x.
%\]
For short we denote this irreducible module by $W_{\left(i\ j\right)}$
with $i>j$. The number of such irreducible modules is $2k^{2}-k$.

When $i=j$, $V_{L+\lambda_{i}}\otimes V_{L+\lambda_{j}}+V_{L+\lambda_{j}}\otimes V_{L+\lambda_{i}}$
split into two different irreducible representations of $\left(V_{L}\otimes V_{L}\right)^{\mathbb{Z}_{2}}$
by \cite{DY}. We denote these irreducible modules by $W_{\widetilde{\left(i\ \epsilon\right)}}$,
$\epsilon=0,1$. The number of such irreducible modules is $4k$.

It is proved in \cite{BDM} that the category of $\sigma$-twisted
$V_{L}\otimes V_{L}$-modules are isomorphic to the category of $V_{L}$-modules,
thus the number of isomorphism classes of irreducible $\sigma$-twisted
$V_{L}\otimes V_{L}$-modules is equal to the number of isomorphism
classes of irreducible $V_{L}$-modules. From \cite{DY}, each twisted
module can also be decomposed into a direct sum of two irreducible
modules of $\left(V_{L}\otimes V_{L}\right)^{\mathbb{Z}_{2}}$. We
denote these irreducible modules by $W_{\widehat{\left(i\ \epsilon\right)}}$,
$i=0,1,\cdots,2k-1$, $\epsilon=0,1$. The number of such irreducible
modules is $4k$.

This gives  $2k^{2}+7k$ irreducible $(V_{L}\otimes V_{L})^{\mathbb{Z}_{2}}$-modules. In fact, all these irreducible modules are inequivalent
as $(V_{L}\otimes V_{L})^{\mathbb{Z}_{2}}$ is a simple current extension of  $V_{\mathbb{Z}\beta}\otimes V_{\mathbb{Z}\beta}^{+}$
and each irreducible module above can be obtained from an irreducible   $V_{\mathbb{Z}\beta}\otimes V_{\mathbb{Z}\beta}^{+}$-module by using the fusion product.

\begin{proposition}\label{all modules}Let $L=\mathbb{Z}\alpha$
be rank one positive definite even lattice with $\left\langle \alpha,\alpha\right\rangle =2k$
where $k$ is some positive integer.Then $\mathcal{U}$ has at most
$2k^{2}+7k$ irreducible modules up to isomorphism as the following
list:

1) For $0\le i\le2k-1,0\le j\le2k-1$, $i>j,$
\[
\left(i\ j\right)=V_{\mathbb{Z}\beta+\frac{i+j}{4k}\beta}\otimes V_{\mathbb{Z}\beta+\frac{i-j}{4k}\beta}+V_{\mathbb{Z}\beta+\frac{2k+i+j}{4k}\beta}\otimes V_{\mathbb{Z}\beta+\frac{2k-i+j}{4k}\beta}.
\]

(2) For all $i\in\left\{ 0,1,\cdots,2k-1\right\} ,$
\[
\left(\widetilde{i\ 0}\right)=V_{\mathbb{Z}\beta+\frac{i}{2k}\beta}\otimes V_{\mathbb{Z}\beta}^{+}+V_{\mathbb{Z}\beta+\frac{i+k}{2k}\beta}\otimes V_{\mathbb{Z}\beta+\frac{\beta}{2}}^{+},
\]
\[
\left(\widetilde{i\ 1}\right)=V_{\mathbb{Z}\beta+\frac{i}{2k}\beta}\otimes V_{\mathbb{Z}\beta}^{-}+V_{\mathbb{Z}\beta+\frac{i+k}{2k}\beta}\otimes V_{\mathbb{Z}\beta+\frac{\beta}{2}}^{-}.
\]

(3) For $0\le i\le2k-1$, when $k+i$ is even,
\[
\left(\widehat{i\ 0}\right)=V_{\mathbb{Z}\beta+\frac{i}{4k}\beta}\otimes(V_{\mathbb{Z}\beta}^{T_{1}})^{+}+V_{\mathbb{Z}\beta+\frac{2k+i}{4k}\beta}\otimes(V_{\mathbb{Z}\beta}^{T_{1}})^{+},
\]
\[
\left(\widehat{i\ 1}\right)=V_{\mathbb{Z}\beta+\frac{i}{4k}\beta}\otimes(V_{\mathbb{Z}\beta}^{T_{1}})^{-}+V_{\mathbb{Z}\beta+\frac{2k+i}{4k}\beta}\otimes(V_{\mathbb{Z}\beta}^{T_{1}})^{-}.
\]

When $k+i$ is odd,

\[
\left(\widehat{i\ 0}\right)=V_{\mathbb{Z}\beta+\frac{i}{4k}\beta}\otimes(V_{\mathbb{Z}\beta}^{T_{2}})^{+}+V_{\mathbb{Z}\beta+\frac{2k+i}{4k}\beta}\otimes(V_{\mathbb{Z}\beta}^{T_{2}})^{-},
\]
\[
\left(\widehat{i\ 1}\right)=V_{\mathbb{Z}\beta+\frac{i}{4k}\beta}\otimes(V_{\mathbb{Z}\beta}^{T_{2}})^{-}+V_{\mathbb{Z}\beta+\frac{2k+i}{4k}\beta}\otimes(V_{\mathbb{Z}\beta}^{T_{2}})^{+}.
\]
\end{proposition}

\begin{proof} Since $V_{\mathbb{Z}\beta}\otimes V_{\mathbb{Z}\beta}^{+}$
is a vertex operator subalgebra of $\mathcal{U}$, by Proposition
\ref{Rationality}, each irreducible $\mathcal{U}$-module $M$ is
a direct sum of irreducible $V_{\mathbb{Z}\beta}\otimes V_{\mathbb{Z}\beta}^{+}$-modules.
Let $W$ be an irreducible $V_{\mathbb{Z}\beta}\otimes V_{\mathbb{Z}\beta}^{+}$-module,
we define
\[
\mathcal{U}\cdot W=\left(V_{\mathbb{Z}\beta}\otimes V_{\mathbb{Z}\beta}^{+}\right)\boxtimes_{V_{\mathbb{Z}\beta}\otimes V_{\mathbb{Z}\beta}^{+}}W\oplus\left(V_{\frac{1}{2}\beta+\mathbb{Z}\beta}\otimes V_{\frac{1}{2}\beta+\mathbb{Z}\beta}^{+}\right)\boxtimes_{V_{\mathbb{Z}\beta}\otimes V_{\mathbb{Z}\beta}^{+}}W.
\]
That is, the fusion product of $\mathcal{U}$ and $W$ as $V_{\mathbb{Z}\beta}\otimes V_{\mathbb{Z}\beta}^{+}$-modules.
Then $\mathcal{U}\cdot W$ is a $\mathcal{U}$-module only if $\mathcal{U}\cdot W$
is $\mathbb{Z}$-graded. First we prove that each module above is
an irreducible $\mathcal{U}$-module.

i) If $W=V_{\mathbb{Z}\beta+\frac{i}{4k}\beta}\otimes V_{\mathbb{Z}\beta+\frac{j}{4k}\beta}$
for some $i\in\left\{ 0,1,\cdots4k-1\right\} $, $j\in\left\{ 1,2\cdots,2k-1\right\} $,
then by fusion rules of irreducible $V_{\mathbb{Z}\beta}$-modules
and $V_{\mathbb{Z}\beta}^{+}$-modules \cite{DL1,A1}
\[
\mathcal{U}\cdot W=V_{\mathbb{Z}\beta+\frac{i}{4k}\beta}\otimes V_{\mathbb{Z}\beta+\frac{j}{4k}\beta}+V_{\mathbb{Z}\beta+\frac{2k+i}{4k}\beta}\otimes V_{\mathbb{Z}\beta+\frac{2k-j}{4k}\beta}.
\]
From which we can see that $\mathcal{U}\cdot W$ is a $\mathcal{U}$-module
only if $i-j\in2\mathbb{Z}$. Denote
\[
\left(i\ j\right)=V_{\mathbb{Z}\beta+\frac{i+j}{4k}\beta}\otimes V_{\mathbb{Z}\beta+\frac{i-j}{4k}\beta}+V_{\mathbb{Z}\beta+\frac{2k+i+j}{4k}\beta}\otimes V_{\mathbb{Z}\beta+\frac{2k-i+j}{4k}\beta},
\]
 then $\{\left(i\ j\right)|0\le i\le2k-1,0\le j\le2k-1,i>j\}$ gives
all $\mathcal{U}\cdot W$ that are irreducible $\mathcal{U}$-modules
up to isomorphism. The number of such modules is $2k^{2}-k$.

ii) If $W=V_{\mathbb{Z}\beta+\frac{j}{4k}\beta}\otimes V_{\mathbb{Z}\beta}^{+}$,
$j=0,1\cdots,4k-1,$ then by fusion rules of irreducible modules for
$V_{\mathbb{Z}\beta}$ and $V_{\mathbb{Z}\beta}^{+}$ \cite{DL1,A1},
\[
\mathcal{U}.W=V_{\mathbb{Z}\beta+\frac{j}{4k}\beta}\otimes V_{\mathbb{Z}\beta}^{+}+V_{\mathbb{Z}\beta+\frac{2k+j}{4k}\beta}\otimes V_{\mathbb{Z}\beta+\frac{1}{2}\beta}^{+}.
\]

So $\mathcal{U}\cdot W$ is a $\mathcal{U}$-module only if $j\in2\mathbb{Z}$.
Let $j=2i$ with $i\in\mathbb{Z}$, then we see that $\left(\widetilde{i\ 0}\right)=V_{\mathbb{Z}\beta+\frac{i}{2k}\beta}\otimes V_{\mathbb{Z}\beta}^{+}+V_{\mathbb{Z}\beta+\frac{i+k}{2k}\beta}\otimes V_{\mathbb{Z}\beta+\frac{\beta}{2}}^{+}$,
$i=0,1,\cdots,2k-1$ are irreducible $\mathcal{U}$-modules.

Similarly, if $W=V_{\mathbb{Z}\beta+\frac{j}{4k}\beta}\otimes V_{\mathbb{Z}\beta}^{-}$,
$V_{\mathbb{Z}\beta+\frac{j}{4k}\beta}\otimes V_{\mathbb{Z}\beta+\frac{\beta}{2}}^{+}$
or $V_{\mathbb{Z}\beta+\frac{j}{4k}\beta}\otimes V_{\mathbb{Z}\beta+\frac{\beta}{2}}^{-}$,
with $j=0,1,\cdots,4k-1$, then $\mathcal{U}\cdot W$ is a $\mathcal{U}$-module
only if $j\in2\mathbb{Z}$. Similarly we let $j=2i$, $i\in\mathbb{Z}$,
denote these $\mathcal{U}\cdot W$ by $\left(\widetilde{i\ 1}\right)$,
$\left(\widetilde{k+i\ 0}\right)$, and $\left(\widetilde{k+i\ 1}\right)$,
$i=0,1,\cdots,2k-1$ respectively.

In this case, $\left\{ \widetilde{\left(i\ \epsilon\right)}|0\le i\le2k-1,\epsilon=0,1\right\} $
gives all $\mathcal{U}\cdot W$ that are irreducible $\mathcal{U}$-modules
up to isomorphism. The number of irreducible $\mathcal{U}$-modules
up to isomorphism is $4k$.

iii) If $W=V_{\mathbb{Z}\beta+\frac{i}{4k}\beta}\otimes V_{\mathbb{Z}\beta}^{T_{1},+}$
for some $i\in\left\{ 0,1,\cdots,4k-1\right\} $, it is easy to check
that if $k+i\in2\mathbb{Z}$, $\mathcal{U}\cdot W=V_{\mathbb{Z}\beta+\frac{i}{4k}\beta}\otimes(V_{\mathbb{Z}\beta}^{T_{1}})^{+}+V_{\mathbb{Z}\beta+\frac{2k+i}{4k}\beta}\otimes(V_{\mathbb{Z}\beta}^{T_{1}})^{+}$
is an irreducible $\mathcal{U}$-module. Denote it by $\left(\widehat{i\ 0}\right)$.
Similarly, we can check that , $V_{\mathbb{Z}\beta+\frac{2k+i}{4k}\beta}\otimes(V_{\mathbb{Z}\beta}^{T_{1}})^{-}+V_{\mathbb{Z}\beta+\frac{i}{4k}\beta}\otimes(V_{\mathbb{Z}\beta}^{T_{1}})^{-}$
is an irreducible $\mathcal{U}$-module only if $k+i$ is even. We
denote it by $\left(\widehat{i\ 1}\right).$ We see $\left\{ \widehat{\left(i\ \epsilon\right)}|0\le i\le2k-1,k+i\in2\mathbb{Z}\right\} $
gives all $\mathcal{U}\cdot W$ that are irreducible $\mathcal{U}$-modules
up to isomorphism.The number of such non-isomorphic irreducible $\mathcal{U}$-module
is $2k$.

By similar argument, we can prove that when $W=V_{\mathbb{Z}\beta+\frac{i}{4k}\beta}\otimes V_{\mathbb{Z}\beta}^{T_{2},\pm}$
for some $i\in\left\{ 0,1,\cdots,4k-1\right\} $, $\mathcal{U}\cdot W$
is an irreducible $\mathcal{U}$-module only if $k+i\in2\mathbb{Z}+1$
. In this case, $\mathcal{U}\cdot W$ is isomorphic to $\widehat{\left(i\ 0\right)}$
and $\widehat{\left(i\ 1\right)}$ respectively as listed above for
the case $k+i$ is odd. The number of irreducible $\mathcal{U}$-modules
up to isomorphism is $2k$.

Since $W$ runs over all irreducible $V_{\mathbb{Z}\beta}\otimes V_{\mathbb{Z}\beta}^{+}$-modules,
from above we see that the number of irreducible $\mathcal{U}$-modules
up to isomorphism is at most $2k^{2}+7k$. \end{proof}

\begin{definition} A vertex opertor algebra $V$ is graded by an
abelian group $D$ if $V=\oplus_{\alpha\in D}V^{\alpha}$, and $u_{n}v\in V^{\alpha+\beta}$
for any $u\in V^{\alpha}$, $v\in V^{\beta}$, and $n\in\mathbb{Z}$.
\end{definition}

\begin{proposition} Let $D$ be an abelian group and $\tilde{V}=\sum_{\alpha\in D}V^{\alpha}$
is a simple $D$-graded vertex operator algebra such that $V^{0}$
is a rational vertex operator subalgebra of $\tilde{V}$ and $V^{\alpha}\not=0$
for all $\alpha\in D$ with $N_{V^{\alpha}V^{\beta}}^{V^{\gamma}}=\delta_{\alpha+\beta,\gamma}$
for all $\alpha,\beta,\gamma\in D$. Let $M=\sum_{\alpha\in D}M^{\alpha}$
be an irreducible $\tilde{V}$-module such that (1) each $ M^{\beta}$ is an irreducible $V^0$-module,
(2) $M^{\alpha}$ and $M^{\beta}$ are inequivalent
for $\alpha\not=\beta,$ and (3) $V^{\alpha}\cdot M^{\beta}=M^{\alpha+\beta}$ for all
$\alpha$ and $\beta.$ Then the $\tilde{V}$-module structure of
$M$ is determined uniquely by the $V^{0}$-module structure of $M$,
that is, if $\left(M,\overline{Y}_{M}\right)$ is also a simple $\tilde{V}$-module
with $\overline{Y}_{M}\left(u,z\right)=Y_{M}\left(u,z\right)$ for
all $u\in V^{0},$ then $\left(M,\overline{Y}_{M}\right)$ and $\left(M,Y_{M}\right)$
are isomorphic. \end{proposition}

\begin{proof} By fusion rules, for any $u\in V^{\alpha},$
$w\in M^{\beta}$, $\alpha,\beta\in D$, there exists nonzero constant
$\lambda_{\alpha,\beta}$ such that $\overline{Y}\left(u,z\right)w=\lambda_{\alpha,\beta}Y\left(u,z\right)w$.
Since $\overline{Y}_{M}\left(u,z\right)=Y_{M}\left(u,z\right)$ for
all $u\in V^{0},$ we have $\lambda_{0},_{\alpha}=1$ for any $\alpha\in D$.

For any $\alpha,\beta,\gamma\in D$, let $u\in V^{\alpha},$ $v\in V^{\beta},$
$w\in M^{\gamma}$, by associativity, there exists nonnegative integer
$k$ such that

\begin{eqnarray*}
\left(z_{1}+z_{2}\right)^{k}\overline{Y}_{M}\left(u,z_{1}\right)\overline{Y}_{M}\left(v,z_{2}\right)w & = & \left(z_{1}+z_{2}\right)^{k}\overline{Y}_{M}\left(Y\left(u,z_{0}\right)v,z_{2}\right)w
\end{eqnarray*}

and that
\[
\left(z_{1}+z_{2}\right)^{k}Y_{M}\left(u,z_{1}\right)Y_{M}\left(v,z_{2}\right)w=\left(z_{1}+z_{2}\right)^{k}Y_{M}\left(Y\left(u,z_{0}\right)v,z_{2}\right)w
\]

Thus we get
\begin{equation}
\lambda_{\alpha,\beta+\gamma}\cdot\lambda_{\beta,\gamma}=\lambda_{\alpha+\beta,\gamma}\label{eq1}
\end{equation}
for any $\alpha,\beta,\gamma\in D$. In particular, let $\gamma=0,$we
obtain $\lambda_{\alpha+\beta,0}=\lambda_{\alpha,\beta}\cdot\lambda_{\beta,0}$.

Let $f:\overline{\sum_{\alpha}M^{\alpha}}\to\sum_{\alpha}M^{\alpha}$
be defined by $f\left(\sum_{\alpha}w_{\alpha}\right)=\sum_{\alpha}\frac{1}{\lambda_{\alpha,0}}w_{\alpha}$
where $w_{\alpha}\in M^{\alpha}$. Then clearly $f$ is a linear isomorphism.
Let $u\in V^{\alpha}$, $w\in M^{\beta},$ then

\begin{eqnarray*}
f\left(\overline{Y}_{M}\left(u,z\right)w\right) & = & f\left(\lambda_{\alpha,\beta}Y\left(u,z\right)w\right)\\
 & = & \frac{1}{\lambda_{\alpha+\beta,0}}\cdot\lambda_{\alpha,\beta}\cdot Y\left(u,z\right)w\\
 & = & \lambda_{\beta,0}Y_{M}\left(u,z\right)w\\
 & = & Y_{M}\left(u,z\right)f\left(w\right).
\end{eqnarray*}
Thus $f$ is an isomorphism of $\tilde{V}$-modules.
\end{proof}

Combining the above two propositons, we get the following results:

\begin{theorem} Let $L=\mathbb{Z}\alpha$ with $\left\langle \alpha,\alpha\right\rangle =2k$
with $k$ some positive integer. Then any irreducible $\left(V_{L}\otimes V_{L}\right)^{\mathbb{Z}_{2}}$-module
is isomorphic to one of the following modules:

\[
\left\{ \left(i\ j\right),\ \widetilde{\left(i\ \epsilon\right)},\ \widehat{\left(i\ \epsilon\right)}|0\le i,j\le2k-1,i>j,\epsilon=0,1\right\}.
\]
\end{theorem}

\section{The quantum dimensions }

The notions and properties of quantum dimensions have been systematically
studied in \cite{DJX}. For a rational, $C_{2}$-cofinite, self-dual
vertex operator algebra of CFT type, quantum dimensions of its irreducible
modules have nice properties. It turns out that these properties are
very helpful in determining fusion rules. Now we recall some notions
and properties about quantum dimensions.

\begin{definition} Let $g$ be an automorphism of the vertex operator
algebra $V$ with order $T$. Let $M=\oplus_{n\in\frac{1}{T}\mathbb{Z}_{+}}M_{\lambda+n}$
be a $g$-twisted $V$-module, the formal character of $M$ is defined
as

\[
\mbox{ch}_{q}M=\mbox{tr}_{M}q^{L\left(0\right)-c/24}=q^{\lambda-c/24}\sum_{n\in\frac{1}{T}\mathbb{Z}_{+}}\left(\dim M_{\lambda+n}\right)q^{n},
\]
where $\lambda$ is the conformal weight of $M$. \end{definition}

We denote the holomorphic function $\mbox{ch}_{q}M$ by $Z_{M}\left(\tau\right)$.
Here and below, $\tau$ is in the upper half plane $\mathbb{H}$ and
$q=e^{2\pi i\tau}$.

\begin{definition} \label{quantum dimension}Let $V$ be a vertex
operator algebra and $M$ a $g$-twisted $V$-module such that $Z_{V}\left(\tau\right)$
and $Z_{M}\left(\tau\right)$ exists. The quantum dimension of $M$
over $V$ is defined as
\[
\mbox{qdim}_{V}M=\lim_{y\to0}\frac{Z_{M}\left(iy\right)}{Z_{V}\left(iy\right)},
\]
where $y$ is real and positive. \end{definition}

From now on, we assume $V$ is a rational, $C_{2}$-cofinite vertex
operator algebra of CFT type with $V\cong V'$. Let $M^{0}\cong V,\, M^{1},\,\cdots,\, M^{d}$
denote all inequivalent irreducible $V$-modules. Moreover, we assume
the conformal weights $\lambda_{i}$ of $M^{i}$ are positive for
all $i>0.$ Then we have the following properties of quantum dimensions
\cite{DJX}:

\begin{proposition} \label{possible values of quantum dimensions}
$q\dim_{V}M^{i}\geq1,$ $\forall i=0,\cdots,d.$

\end{proposition}

\begin{proposition} \label{quantum-product} For any $i,\, j=0,\cdots,\, d,$
\[
q\dim_{V}\left(M^{i}\boxtimes M^{j}\right)=q\dim_{V}M^{i}\cdot q\dim_{V}M^{j}.
\]
\end{proposition}

\begin{proposition} \label{simple current quantum dim}A $V$-module
$M$ is a simple current if and only if $q\dim_{V}M=1$.\end{proposition}

\begin{remark} By Proposition \ref{Rationality} and \cite{A4} we
see that the vertex operator algebra $\left(V_{L}\otimes V_{L}\right)^{\mathbb{Z}_{2}}$
satisfies all the assumptions and thus we can apply these properties.\end{remark}

By applying properties of irreducible $\left(V_{L}\otimes V_{L}\right)^{\mathbb{Z}_{2}}$-modules,
we obtain quantum dimensions of all irreducible $\left(V_{L}\otimes V_{L}\right)^{\mathbb{Z}_{2}}$-modules
as follows:

\begin{proposition}For $0\le i,j\le2k-1$,$i>j$, $\epsilon=0,1$,
\[
q\dim\left(ij\right)=2,
\]
\[
q\dim\left(\widetilde{i\ \epsilon}\right)=1,
\]
\[
q\dim\left(\widehat{i\ \epsilon}\right)=\sqrt{2k}.
\]

\end{proposition}

\begin{proof} First we notice that by fusion rules of irreducible
modules of $V_{\mathbb{Z}\beta}$ and $V_{\mathbb{Z}\beta}^{+}$,
$V_{\frac{1}{2}\beta+\mathbb{Z}\beta}\otimes V_{\frac{1}{2}\beta+\mathbb{Z}\beta}^{+}$
is a simple current. Thus $q\dim_{V_{\mathbb{Z}\beta}\otimes V_{\mathbb{Z}\beta}^{+}}\mathcal{U}=2$.
By definition of quantum dimension, we see that for each $\mathcal{U}$-module
$W$,
\[
q\dim_{\mathcal{U}}W=\frac{q\dim_{V_{\mathbb{Z}\beta}\otimes V_{\mathbb{Z}\beta}^{+}}W}{q\dim_{V_{\mathbb{Z}\beta}\otimes V_{\mathbb{Z}\beta}^{+}}\mathcal{U}}=\frac{1}{2}q\dim_{V_{\mathbb{Z}\beta}\otimes V_{\mathbb{Z}\beta}^{+}}W.
\]
Thus it suffices to find quantum dimensions of these modules as $V_{\mathbb{Z}\beta}\otimes V_{\mathbb{Z}\beta}^{+}$-
modules. By fusion rules of irreducible $V_{\mathbb{Z}\beta}$-modules,
we see that $V_{\mathbb{Z}\beta+\frac{i}{4k}\beta}$ is a simple current.
By Proposition \ref{simple current quantum dim}, $q\dim_{V_{\mathbb{Z}\beta}}V_{\mathbb{Z}\beta+\frac{i}{4k}\beta}=1$.
By fusion rules of irreducible modules, both $V_{\mathbb{Z}\beta}^{+}$
and $V_{\mathbb{Z}\beta}^{-}$ are simple currents of $V_{\mathbb{Z}\beta}^{+}$,
therefore $q\dim_{V_{\mathbb{Z}\beta}^{+}}V_{\mathbb{Z}\beta}^{+}=q\dim_{V_{\mathbb{Z}\beta}^{+}}V_{\mathbb{Z}\beta}^{-}=1$.
Since $V_{\mathbb{Z}\beta}=V_{\mathbb{Z}\beta}^{+}+V_{\mathbb{Z}\beta}^{-}$,
we have $q\dim_{V_{\mathbb{Z}\beta}^{+}}V_{\mathbb{Z}\beta}=2$. Similarly
we get $q\dim_{V_{\mathbb{Z}\beta}^{+}}V_{\mathbb{Z}\beta+\frac{i}{4k}\beta}=2$,
$i\in\left\{ 1,\cdots,2k-1\right\} $. By Proposition \ref{quantum-product},
we obtain
\[
q\dim_{V_{\mathbb{Z}\beta}\otimes V_{\mathbb{Z}\beta}^{+}}\left(i\ j\right)=4
\]
and therefore $q\dim_{\mathcal{U}}\left(ij\right)=2$, $0\le i,j\le2k-1$,
$i>j$.

By similar argument, it is easy to see that $q\dim\left(\widetilde{i\ \epsilon}\right)=1$ for
$i=0,1,\cdots,2k-1$, $\epsilon=0,1$.

To obtain the quantum dimensions of $\widehat{\left(i\ \epsilon\right)}$,
we first observe that by fusion rules of irreducible $V_{\mathbb{Z}\beta}^{+}$-modules
we have the fusion product:
\[
V_{\mathbb{Z}\beta}^{T_{2},-}\boxtimes_{V_{\mathbb{Z}\beta}^{+}}V_{\mathbb{Z}\beta}^{T_{2},-}=V_{\mathbb{Z}\beta}^{+}+V_{\frac{\beta}{2}+\mathbb{Z}\beta}^{-}+\sum_{1\le r\le2k-1,r\ \mbox{even}}V_{\frac{r}{4k}+\mathbb{Z}\beta}.
\]
Since $q\dim_{V_{\mathbb{Z}\beta}^{+}}V_{\frac{r}{4k}+\mathbb{Z}\beta}=2,$
we see that the quantum dimension of the right hand side is $2k$ and
\[
q\dim\left(V_{\mathbb{Z}\beta}^{T_{2},-}\boxtimes_{V_{\mathbb{Z}\beta}^{+}}V_{\mathbb{Z}\beta}^{T_{2},-}\right)=2k.
\]
By Proposition \ref{quantum-product}, we obtain $q\dim_{V_{\mathbb{Z}\beta}^{+}}V_{\mathbb{Z}\beta}^{T_{2},-}=\sqrt{2k}$.
Similarly, we can prove
\[
q\dim_{V_{\mathbb{Z}\beta}^{+}}V_{\mathbb{Z}\beta}^{T_{1},+}=q\dim_{V_{\mathbb{Z}\beta}^{+}}V_{\mathbb{Z}\beta}^{T_{1},-}=q\dim_{V_{\mathbb{Z}\beta}^{+}}V_{\mathbb{Z}\beta}^{T_{2},+}=q\dim_{V_{\mathbb{Z}\beta}^{+}}V_{\mathbb{Z}\beta}^{T_{2},-}=\sqrt{2k}.
\]
Thus $q\dim_{V_{\mathbb{Z}\beta}\otimes V_{\mathbb{Z}\beta}^{+}}\widehat{\left(i\ \epsilon\right)}=2\sqrt{2k}$
and $q\dim_{\mathcal{U}}\widehat{\left(i\ \epsilon\right)}=\sqrt{2k}$
, $i=0,1,\cdots,2k$, $\epsilon=0,1$. \end{proof}

\section{Fusion Rules}

In this section, we apply the quantum dimensions obtained in the
previous section and the fusion rules of irreducible $V_{\mathbb{Z}\beta}$-
and $V_{\mathbb{Z}\beta}^{+}$-modules in \cite{DL1} and \cite{A2}  to determine the fusion products of
the cyclic permutation orbifold model.

\begin{theorem} For any $i>j$, $0\le i,j,l\le2k-1$, $\epsilon\in\left\{ 0,1\right\} $,
the following are the fusion products of irreducible modules of $\left(V_{L}\otimes V_{L}\right)^{\mathbb{Z}_{2}}$
:

\begin{equation}
\left(i\ j\right)\boxtimes\widehat{\left(l\ \epsilon\right)}=\widehat{\left(i+j+l\ 0\right)}+\widehat{\left(i+j+l\ 1\right)}\label{eq:non-diagonal-twisted}
\end{equation}
where $i+j+l$ is understood to be its remainder in $\{0,...,2k-1\}$ when divided by $2k,$

\begin{equation}
\left(i\ j\right)\boxtimes\widetilde{\left(l\ \epsilon\right)}=\left(i+l\ j+l\right),\label{non-diagonal-diagonal}
\end{equation}

\begin{equation}
\widetilde{\left(i\ \epsilon\right)}\boxtimes\widetilde{\left(j\ \epsilon_{1}\right)}=\widetilde{\left(i+j\ \epsilon+\epsilon_{1}\right)}\label{diagonal-diagonal}
\end{equation}
where $\epsilon+\epsilon_{1}$ is understood to be its remainder in $\{0,1\}$ when divided by $2,$

\begin{equation}
\widetilde{\left(i\ \epsilon\right)}\boxtimes\widehat{\left(j\ \epsilon_{1}\right)}=\widehat{\left(2i+j\ \epsilon+\epsilon_{1}\right)},\label{diagonal-twisted}
\end{equation}

\begin{equation}
\left(i\ j\right)\boxtimes\left(p\ q\right)=\widetilde{\left(p+j\ 0\right)}+\widetilde{\left(p+j\ 1\right)}+\widetilde{\left(p-i\ 0\right)}+\widetilde{\left(p-i\ 1\right)}\label{eq:nondiagonal-nondiagonal}
\end{equation}
if $i-j=p-q,$

\[
\left(i\ j\right)\boxtimes\left(p\ q\right)=\left(i+p\ j+q\right)+\left(i+q\ j-p\right)
\]
if $i-j\pm\left(p-q\right)\not=0,2k$,

\begin{equation}
\widehat{\left(i\ \epsilon\right)}\boxtimes\widehat{\left(j\ \epsilon_{1}\right)}=\widetilde{\left(\frac{i+j}{2}\ \epsilon+\epsilon_{1}\right)}+\widetilde{\left(k+\frac{i+j}{2}\ \epsilon+\epsilon_{1}+1\right)}+\sum_{1\le r\le2k-1,r\ \mbox{even}}\left(i+r\ j-r\right)\label{twisted-twisted:one odd two even}
\end{equation}
if $k+i$ and $k+j$ are both odd,

\begin{equation}
\widehat{\left(i\ \epsilon\right)}\boxtimes\widehat{\left(j\ \epsilon_{1}\right)}=\widetilde{\left(\frac{i+j}{2}\ \epsilon+\epsilon_{1}\right)}+\widetilde{\left(k+\frac{i+j}{2}\ \epsilon+\epsilon_{1}\right)}+\sum_{1\le r\le2k-1,r\ \mbox{even}}\left(i+r\ j-r\right)\label{twisted-twisted:all odd or all even}
\end{equation}
if $k+i$ and $k+j$ are both even, and
\begin{equation}
\widehat{\left(i\ \epsilon\right)}\boxtimes\widehat{\left(j\ \epsilon_{1}\right)}=\sum_{1\le r\le2k-1,r\ \mbox{odd }}\left(i+r\ j-r\right)\label{twisted-twisted:one odd one even}
\end{equation}
if one of  $k+i$ and $k+j$ is even and the other is odd.
\end{theorem}

\begin{proof}\emph{Proof of (\ref{eq:non-diagonal-twisted}), (\ref{non-diagonal-diagonal}),
(\ref{diagonal-diagonal}) and (\ref{diagonal-twisted}):} For $0\le i,j\le2k-1$
with $i>j$, first by fusion rules of irreducible $V_{\mathbb{Z}\beta}$-
and $V_{\mathbb{Z}\beta}^{+}$- modules, we get $I_{\mathcal{U}}\left(_{\left(ij\right)\ \widehat{\left(l\ \epsilon\right)}}^{\widehat{\left(i+j+l\ \epsilon\right)}}\right)\not=0$,
$\epsilon=0,1$. By counting quantum dimensions of both sides, we
get (\ref{eq:non-diagonal-twisted}). By similar arguments, we can
prove (\ref{non-diagonal-diagonal})\emph{,} (\ref{diagonal-diagonal})
and (\ref{diagonal-twisted})\emph{. }

\emph{Proof of (\ref{eq:nondiagonal-nondiagonal}): }For $0\le i,j,p,q\le2k-1$
with $i>j,p>q$ and $i-j=p-q$, by fusion rules of irreducible $V_{\mathbb{Z}\beta}$
and $V_{\mathbb{Z}\beta}^{+}$ modules, we have $I_{\mathcal{U}}\left(_{\left(i\ j\right)\ \left(p\ q\right)}^{\widetilde{\left(\frac{i+j+p+q}{2}\ \epsilon\right)}}\right)\not=0$,
$\epsilon=0,1$. Since $i-j=p-q$, we write $\frac{i+j+p+q}{2}$ as
$p+j$, then $I_{\mathcal{U}}\left(_{\left(i\ j\right)\ \left(p\ q\right)}^{\widetilde{\left(p+j\ \epsilon\right)}}\right)\not=0$.
Besides, by (\ref{non-diagonal-diagonal}) and Proposition \ref{fusion rule symmmetry property},
we get $I_{\mathcal{U}}\left(_{\left(i\ j\right)\ \left(p\ q\right)}^{\widetilde{\left(p-i\ \epsilon\right)}}\right)\not=0$
for $\epsilon=0,1$. By counting quantum dimensions of both sides,
we get (\ref{eq:nondiagonal-nondiagonal}).

\emph{Proof of (\ref{twisted-twisted:one odd two even}) and (\ref{twisted-twisted:all odd or all even}): }

When $k+i$ and $k+j$ are both odd, by fusion rules of irreducible
$V_{\mathbb{Z}\beta}$ and $V_{\mathbb{Z}\beta}^{+}$ modules, we
get $I_{\mathcal{U}}\left(_{\widehat{\left(i\ \epsilon\right)}\ \widehat{\left(j\ \epsilon_{1}\right)}}^{\widetilde{\left(\frac{i+j}{2}\ \epsilon+\epsilon_{1}\right)}}\right)\not=0$,
$I_{\mathcal{U}}\left(_{\widehat{\left(i\ \epsilon\right)}\ \widehat{\left(j\ \epsilon_{1}\right)}}^{\widetilde{\left(k+\frac{i+j}{2}\ \epsilon+\epsilon_{1}+1\right)}}\right)\not=0$,
$\epsilon,\epsilon_{1}=0,1$. For $1\le r\le2k-1$ with $r$ even,
$ $ notice that $I_{V_{\mathbb{Z}\beta}^{+}}\left(_{V_{\mathbb{Z}\beta}^{T_{2},\pm}\ V_{\mathbb{Z}\beta}^{T_{2},\pm}}^{^{V_{\frac{r}{4k}\beta+\mathbb{Z}\beta}}}\right)\not=0$,
we obtain $I_{\mathcal{U}}\left(_{\widehat{\left(i\ \epsilon\right)}\ \widehat{\left(j\ \epsilon_{1}\right)}}^{\widetilde{\left(i+r\ j-r\right)}}\right)\not=0$.
By counting quantum dimensions of both sides, we get\emph{ }(\ref{twisted-twisted:one odd two even}).
By similar argument, we can prove (\ref{twisted-twisted:all odd or all even}).

\emph{Proof of} \emph{(\ref{twisted-twisted:one odd one even}): }By
fusion rules of irreducible $V_{\mathbb{Z}\beta}^{+}$-modules, we
see that $I_{V_{\mathbb{Z}\beta}^{+}}\left(_{V_{\mathbb{Z}\beta}^{T_{2},\pm}\ V_{\mathbb{Z}\beta}^{T_{1},\pm}}^{^{V_{\frac{r}{4k}\beta+\mathbb{Z}\beta}}}\right)$ is nonzero
for $1\le r\le2k-1$ with $r$ odd. By observing irreducible $\mathcal{U}$-modules
as listed in Proposition \ref{all modules} and applying fusion
rules of irreducible $V_{\mathbb{Z}\beta}$-modules, we obtain $I_{\mathcal{U}}\left(_{\widehat{\left(i\ \epsilon\right)}\ \widehat{\left(j\ \epsilon_{1}\right)}}^{\widetilde{\left(i+r\ j-r\right)}}\right)\not=0$.
By counting quantum dimensions, we prove the fusion product (\ref{twisted-twisted:one odd one even}).

\end{proof}

\end{document}